\RequirePackage{fix-cm}
\documentclass[smallextended]{svjour3}       
\smartqed  
\usepackage{graphicx}
\usepackage{amsfonts}
\usepackage{amsmath}
\usepackage{amssymb}
\usepackage{graphicx}
\usepackage{latexsym}
\usepackage{color}
\usepackage{mathrsfs}

\newcommand{\s}{\mathcal{S}}
\newcommand{\Li}{\mathcal{L}}

\begin{document}

\title{A geometric proof of the Periodic Averaging Theorem on Riemannian manifolds}



\author{Misael Avenda\~{n}o Camacho   \and \\ Guillermo D\'avila Rasc\'on }


\institute{ Misael Avenda\~{n}o Camacho 
            \at CONACYT Research Fellow-Departamento de Matem\'aticas, Universidad de Sonora \\
            Blvd L. Encinas y Rosales s/n Col. Centro, CP 83000, Hermosillo (Son), Mexico \\
              Tel.: +52-66-2592155 \\
              \email{misaelave@mat.uson.mx}
              \and
              Guillermo D\'avila Rasc\'on
              \at Departamento de Matem\'aticas, Universidad de Sonora \\
            Blvd L. Encinas y Rosales s/n Col. Centro, CP 83000, Hermosillo (Son), Mexico \\
              \email{davila@mat.uson.mx}           
           }

\date{Received: date / Accepted: date}

\maketitle

\begin{abstract}
We present a geometric proof of the averaging theorem for perturbed dynamical systems on a Riemannian manifold,
in the case where the flow of the unperturbed vector field is periodic and the $\mathbb{S}^{1}$-action associated to this vector
field is not necessarily trivial. We generalize the averaging procedure \cite{Arno-63,ArKN-88} defining a global averaging method
based on a free coordinate approach which allow us to formulate our results on any open domain with compact closure.
\keywords{Averaging method \and Perturbation theory \and Periodic flows \and Riemannian manifolds \and Horizontal lifts \and
$\mathbb{S}^1$-principal bundle}
 \subclass{37C10 \and 37C55 \and 37J40 \and 53C20 \and 55R10 \and 58D17}
\end{abstract}

\section{Introduction}
\label{intro}
The well known averaging method \cite{Arno-63,ArKN-88,Mos-70,Ne-08,Reeb,verhulst3} is one of the most important methods in perturbation theory
and it is based on the idea of splitting the motion of a perturbed system into a slow evolution and rapid oscillations.

Geometrically, the averaging method arises in the context of perturbations of vertical vector fields on a fibered
manifold. We consider a smooth fiber bundle $\pi : M \rightarrow B$ and a smooth, perturbed vector
field $A_{\varepsilon} = A_{0} + \varepsilon A_{1}$ on $M$, where $A_{0}$ is tangent to every fiber. In this situation,
each integral curve of $A_{0}$ is projected by $\pi$ onto a point on the base $B$ and each integral curve of $A_{\varepsilon}$ is
projected into a curve on $B$, whose tangent vector field is of order $\varepsilon$ but, in general, that projected curve
is not the integral curve of any vector field on the base $B$. Therefore, a noticeable displacement of the projected curve takes place
over time of order $1/\varepsilon$.  This situation rises the following question:
Is it possible to describe the motion of the projected curve  on the base, for a long period of time? The
averaging method allow us to describe the motion of this projected curve by means of the integral curve of a certain
vector field on the base $B$.

In many applications of the averaging method, the fiber bundles have the following properties: (i) every point of the base has a neighborhood
where the fiber bundle is a direct product, and (ii) the fibers are $n$-dimensional tori. However, the only case
completely studied is when the fibers are one dimensional tori (circles), the so-called one frequency systems, \cite{ArKN-88,verhulst3}.

Consider the product manifold $\mathbb{S}^{1} \times \mathbb{R}^{n}$ together with the coordinate system  $(\varphi \, (\mathrm{mod} \, 2\pi), I)$. The one frequency system is the perturbed dynamical system on $\mathbb{S}^{1} \times \mathbb{R}^{n}$ given by
\begin{align}
\dot{\varphi}& = \omega(I) + \varepsilon f(\varphi, I),  \label{pert1} \\
\dot{I}& = \varepsilon \, g(\varphi, I),    \label{pert2}
\end{align}
for $0 \leq \varepsilon \ll 1$, where $f = f(\varphi, I)$ and $g=g(\varphi, I)$ are smooth $2\pi$-periodic
in $\varphi$, and $\omega: \mathbb{R}^n \rightarrow \mathbb{R}$  is called \emph{the frequency function}. Notice that for
$\varepsilon =0$,  that is, with no perturbation, the system (\ref{pert1}), (\ref{pert2}) has periodic solutions with
frequency $\omega$. If we consider the canonical projection $\pi: \mathbb{S}^{1} \times \mathbb{R}^{n}\rightarrow \mathbb{R}^{n} $ over the second factor, we have a trivial fiber bundle with $\mathbb{S}^1$ as typical fibers and $\mathbb{R}^{n}$ as the base. Therefore, for $\varepsilon = 0$, the unperturbed system of (\ref{pert1}), (\ref{pert2}) defines a vector field which is
tangent to the fibers. In order to approximate the projection of the trajectories of the one frequency system (\ref{pert1}), (\ref{pert2}) over the base $\mathbb{R}^{n}$, the averaging procedure suggests to replace the \emph{slow part} (\ref{pert2}) by
the \emph{averaged system}

\begin{equation}
\dot{J}  = \varepsilon \, G(J),   \label{avertraj}
\end{equation}
where
\[
G(J) = \frac{1}{2\pi}\int^{2\pi}_0  g(\varphi,J) \mathrm{d} \varphi,
\]
and then compare the trajectories of (\ref{pert2}) and (\ref{avertraj}) at the same initial condition. Remark that equation (\ref{avertraj}) has the advantage that does not depend on the coordinate $\varphi$.

The classical averaging theorem asserts that if the frequency function satisfies the non degeneracy condition:
$\omega(I) > c^{-1} >0 $ for a certain constant $c$, then the solution $J(t)$ of the averaged system (\ref{avertraj}) remains close enough to
the solution $I(t)$ of slow part (\ref{pert2}), with $I(0)=J(0)$, for $\varepsilon$ small enough on the long time scale
$t \sim 1/\varepsilon$, that is, there exist constants $c_1>0$ and $\varepsilon_0 >0$ such that
\begin{equation*}
\|I(t)-J(t)  \| < c_1\varepsilon \ \ \ \text{ if } \ \ \ I(0)=J(0) \ \ \ \text{and}\ \ \ 0 \leq t \leq \frac{1}{\varepsilon},
\end{equation*}
for all $\varepsilon < \varepsilon_0$, \cite{Arno-63,ArKN-88,Ne-08,verhulst3}.

The purpose of this paper is to generalize the classical averaging theorem in the following
setting: instead of the frequency system (\ref{pert1}), (\ref{pert2}) on $\mathbb{S}^{1} \times \mathbb{R}^{n}$, we consider
a smooth, perturbed vector field $X_\varepsilon = X_0 + \varepsilon X_1$ on an arbitrary manifold $M$, where
$X_0$ is a vector field with periodic flow, and hence, the vector field $X_0$ induces an $\mathbb{S}^1$-action on $M$.
Now, assume that this action is free and let $\mathcal{O} = M/\mathbb{S}^1$ be the orbit space. Thus, $\mathcal{O}$ is a smooth manifold and
the natural projection $\rho: M \rightarrow \mathcal{O}$ is a surjective submersion and hence, we have a fiber bundle $(M, \rho, \mathcal{O})$
having $\mathbb{S}^1$ as typical fiber. Therefore, we are interested in finding some estimates for the projections of the
trajectories of the perturbed vector field $X_\varepsilon$ over $\mathcal{O}$ with respect to the trajectories of a
suitable vector filed on the base.

If we assume that the fiber bundle $\rho: M \rightarrow \mathcal{O}$ is locally trivial, there exists a local
coordinate system where the perturbed vector field  $X_\varepsilon$ takes the form (\ref{pert1}), (\ref{pert2}). We can try to apply the classical averaging theorem in this setting. However, this approach has a major drawback: it could happen that the
integral curves of the perturbed vector field should not be completely contained in the local coordinate system or, else, it could happen that they pass through it only for a short period of time. Since the averaging theorem applies only on
the coordinate neighborhood where the perturbed vector field  $X_\varepsilon$ takes the form (\ref{pert1}), (\ref{pert2}), we do not know what occurs outside of this  neighborhood. Therefore, we are not able to obtain an approximation of the projected trajectories of $X_\varepsilon$ for a long period of time. This drawback does not occur if the coordinates are well defined on the whole manifold $M$, but the existence of global action-angle coordinates is a very restrictive situation
\cite{bates-92,Duist-80,Ne-08,Nekh88}.

In this paper, we prove the averaging theorem on arbitrary manifolds without the assumption of the existence of special coordinate systems
like action-angle variables. Actually, we follow a coordinate free approach. We study perturbations of vector
fields with periodic flow on arbitrary open domains of a Riemannian manifold $M$ when the $\mathbb{S}^{1}$-action associated to the unperturbed
vector field is not necessarily trivial. Here, we define a global averaging method using the properties of periodic flows which allow us to
formulate our results in a global setting.

The proof of the classical averaging theorem \cite{ArKN-88,verhulst3} is based on the following arguments: a near identity transformation whose infinitesimal generator is a solution of a homological type equation, the triangle inequality and some technical estimations (for example,
Gronwall type estimation are presented in \cite{verhulst3}). In our setting, the proof of the theorem also follows from the same arguments;
however, we face with some difficulties which are not present in the classical formulation of the averaging theorem. The main of these
difficulties is to get an inequality of Gronwall's type which help us find an estimation between the distance of the perturbed trajectory
and the averaged trajectory. In the classical setting, this estimation is obtained due to the existence of global minimizing geodesic on
$\mathbb{S}^{1} \times \mathbb{R}^n$. However, this property does not hold in general. To address this problem we use a geometric construction.
The idea here is to construct a parameterized surface $\gamma: [0,1] \times [0,L] \rightarrow M$ given by
$\gamma(s,t):=\mathrm{Fl}^t_{X_S}(\beta(s))$ where $\beta: [0,1] \rightarrow M $ is a fixed curve and $\mathrm{Fl}^t_{X_S}$ is the flow of a
parameterized vector field $X_s$. Then, for every $t \in [0,L]$ we estimate the distance from $\gamma(0,t)$ and $\gamma(1,t)$ by
using Gronwall's lemma \cite{kunzi,verhulst3} and assuming that the manifold $M$ possesses a
suitable Riemannian metric, so we can use such tools like covariant derivatives and horizontal lifts.

The paper is organized as follows. In Sec. 2, we introduce the averaging and integrating operator associated to a vector field with periodic flow.
In Sec. 3, we state our hypotheses and main result. Then, in Sec. 4 we show how to construct a near identity transformation putting the perturbed
vector field $X_\varepsilon = X_0 + \varepsilon X_1$ into its $\mathbb{S}^1$-invariant normal form relative to the $\mathbb{S}^1$-action induced
by $X_0$. To achieve this goal, a kind of homological equation must be solved. In Sec.~5, we state a Gronwall's type
inequality on Riemannian manifolds. In Sec. 6, we define an $\mathbb{S}^1$-invariant horizontal distribution on $TM$ using the fact that in the
$\mathbb{S}^1$ bundle $(\mathbb{S}^1, \rho, \mathcal{O} )$  the map $\rho $ is Riemannian submersion. We also show the basic properties of the
horizontal lifts of curves and vector fields. Sec.~7 is devoted to the proof of the main theorem.
Finally, in Sec.~8 we  make use of the averaging theorem (Theorem \ref{aver_theo}) in order to
construct adiabatic invariants for perturbed vector fields.

\section{Averaging operators associated to periodic flows}
\label{sec:2}

Let $M$ be a smooth manifold and let $X_0$ be a complete vector field on $M$ with periodic flow
$\mathrm{Fl}_{X_0}^{t}$ and frequency function $\omega: M\rightarrow\mathbb{R}$, $\omega >0$ , that is, for
any $p \in M$
\begin{equation} \label{flow:X}
\mathrm{Fl}_{X_0}^{t+T(p)}(p)= \mathrm{Fl}_{X_0}^{t}(p),  \qquad \forall \; t \in \mathbb{R},
\end{equation}
where $T(p) := {2\pi}/{\omega(p)}$ is the period of the orbit of $X_0$ passing through $p$. The vector field
$X_0$ induces an $\mathbb{S}^{1}$-action on $M$ given by $(t,p)\rightarrow  \mathrm{Fl}_{X_0}^{\frac{1}{\omega(p)}t}(p)$, with coordinate
$t\ \mathrm{mod}\, 2\pi$. We denote the infinitesimal generator of the $\mathbb{S}^{1}$-action by $\mathrm{\Upsilon}$, which can be computed in
terms of the vector field $X_0$, and is given by
\begin{equation}  \label{infgen}%
\mathrm{\Upsilon} = \frac{1}{\omega}X_0.
\end{equation}

Now, in order to settle the main result of the paper, let us recall some useful facts. A tensor field $R \in \Gamma T_{r}^{s}(M)$
is said to be $\mathbb{S}^{1}$-\emph{invariant} if and only if $\left(\mathrm{Fl}_{\mathrm{\Upsilon}}^{t}\right)^{\ast} R = R $. Equivalently,
$\Li_\mathrm{\Upsilon} R = 0$, where $\mathcal{L}$ denotes the usual Lie derivative. In most cases, we will be using this definition for smooth
functions (tensor fields in  $\Gamma T_{0}^{0}(M)$) and smooth vector fields  (tensor fields in  $\Gamma T_{0}^{1}(M)$).

For any tensor field $R \in \Gamma T_{r}^{s}(M)$, we associate, with the $\mathbb{S}^{1}$-action on $M$, defined by
$X_0$, the following operators acting on $T_{r}^{s}(M)$:
\begin{enumerate}
\item The \emph{averaging operator} which is the tensor field defined by
\begin{equation}   \label{avgoper}%
\langle R \rangle := \frac{1}{2\pi}\int^{2\pi}_0 (\mathrm{Fl}_{\mathrm{\Upsilon}}^t)^{\ast} R.
\end{equation}
Notice that $\langle R \rangle$ is a tensor field of the same type as $R$.\\
\item The \emph{integrating operator}, which is the tensor field defined by
\begin{equation} \label{intoper}%
\mathcal{S}(R) := \frac{1}{2\pi} \int_{0}^{2\pi} (t -\pi)(\mathrm{Fl}_{\mathrm{\Upsilon}}^{t})^{\ast} R.
\end{equation}
It is clear that $\mathcal{S}(R)$ is a tensor field of the same type as $R$.
\end{enumerate}

For every $R \in \Gamma T_{r}^{s} (M)$, the operators defined in (\ref{avgoper}) and (\ref{intoper})  have
the following properties \cite{MVor-11,Mos-70}:
\begin{enumerate}
\item $\Li_{\mathrm{\Upsilon}}\langle R \rangle =0$.
\item $R$ is $\mathbb{S}^{1}$-invariant  if and only if $\langle R \rangle = R$.
\item $\langle \Li_\mathrm{\Upsilon} R \rangle=\Li_{\mathrm{\Upsilon}}\langle R \rangle$.
\item For an   $\mathbb{S}^{1}$-invariant function $g \in C^\infty(M)$ we have
$\langle gR \rangle = g\langle R \rangle$ and $\mathcal{S}(gR) = g\mathcal{S}(R)$.
\item $\mathcal{S}(\langle R \rangle) = \langle \mathcal{S}(R) \rangle = 0$
\item $\Li_\mathrm{\Upsilon} \circ \mathcal{S} (R) = R - \langle R \rangle$.
\end{enumerate}

A key property relating the averaging operator and the integrating operator is given by the following result.
\begin{proposition}\label{homol_prop}
For every $R \in \mathfrak{X}(M)$, the vector field
$Z=\frac{1}{\omega}\mathcal{S}(R) +\frac{1}{\omega^{3}}\mathcal{S}^2(\Li_R \omega)X_0 $ satisfies the homological type equation
\begin{equation}
\Li_{X_0}  Z = R - \langle R \rangle.
\end{equation}
\end{proposition}
The proof of this proposition follows from properties 1-6, above (see \cite{MVor-11}).

\section{Main result: The periodic averaging theorem}
\label{sec:3}

Let $ M$ be a connected manifold and let $X_0$ be a vector field on $M$. We assume that the vector field $X_0$ satisfies the following symmetry hypothesis:
\begin{enumerate}
\item[(SH)] $X_{0}$ is a vector field on $M$ with periodic flow and the action of the circle
$\mathbb{S}^{1} =\mathbb{R\diagup} 2\pi\mathbb{Z}$ on $M$ associated with the infinitesimal generator $\mathrm{\Upsilon} =\frac{1}{\omega}X_0$
is \textit{free}.

\end{enumerate}
Let $\mathcal{O}= M\mathbb{\diagup S}^{1}$ be the orbit space of the $\mathbb{S}^{1}$-action and denote by $\rho : M\rightarrow\mathcal{O}$ the
natural  projection. It follows from well-known properties of free actions of compact Lie groups \cite{marsden2,michor} that there exists a
unique manifold structure on $\mathcal{O}$  such that $\rho$ is a smooth surjective submersion (a fiber bundle). Moreover, $\rho$ is a principal
$\mathbb{S}^{1}$-bundle over $\mathcal{O}$. For each $\mathbb{S}^{1}$-invariant vector field $Y$ on $M$ there exists a unique vector field
$Y_{\mathcal{O}}$ on the orbit space $\mathcal{O}$ which is $\rho$-related with $Y$, that is,
\[
\rho_* Y =Y_{\mathcal{O}} .
\]
It is clear that $\mathrm{\Upsilon}_{\mathcal{O}}\equiv 0$.

Now, let us choose an $\mathbb{S}^{1}$-invariant metric $g$ on $M$. Such a Riemannian metric always
exists and can be obtained from an arbitrary Riemannian metric on $M$ by applying the averaging procedure, \cite{MVor-11}.
Since the $\mathbb{S}^{1}$-action is free and proper, there exists a unique Riemannian metric on $\mathcal{O}$,
denoted by $g^{\mathcal{O}}$, such that the projection $\rho: M\rightarrow \mathcal{O}$ is a
Riemannian submersion \cite{gallot,michor}.  We also denote by
$\mathrm{dist}^{\mathcal{O}}:\mathcal{O}\times\mathcal{O}\rightarrow\mathbb{R}$ the corresponding distance function.

In order to formulate the main result of this paper, we start with a perturbed vector field on $M$, which is close
to the smooth vector field $X_0$ in the following sense:
\begin{equation} \label{pertvfX} %
X_{\varepsilon} = X_{0} +\varepsilon X_{1}, \qquad \varepsilon \ll 1.
\end{equation}
As usual,   $X_{1}$ is also a smooth vector fields on $M$, known as the
\emph{perturbed} part of $X_{\varepsilon}$.
The averaging theorem establishes that the projection over the orbit space $\mathcal{O}$ of a trajectory of
$X_{\varepsilon}$ (\ref{pertvfX}) can be approximated by a  trajectory of the vector field
$\varepsilon \langle X_1\rangle_\mathcal{O}$, where $\langle X_1\rangle_\mathcal{O}=\rho_*\langle X_1\rangle $.

\begin{theorem}[Periodic Averaging Theorem on Manifolds] \label{aver_theo}
Let $ M$ be a connected manifold. Assume that the vector field $X_0$ satisfies the symmetry hypothesis {\upshape (SH)} above. Fix $m^{0}\in M$
and suppose that there exists a constant $L>0$ such that the trajectory of $\langle X_{1}\rangle_{\mathcal{O}} \in \mathfrak{X}(\mathcal{O})$ through
$z^{0}=\rho(m^{0})\in\mathcal{O}$ is defined for all $t\in [0,L]$ and remains in an open domain $\mathcal{D}_{0}$
with compact closure. Then, there exist constants $\varepsilon_{0}>0$, $L \geq L_{0}>0$ and $c>0$ such that
\begin{equation}  \label{CCC}
\mathrm{dist}^{\mathcal{O}} \Bigl( \rho\circ\mathrm{Fl}_{\mathbf{X}_{\varepsilon}}^{t} (m^{0}),
\mathrm{Fl}_{\langle X_{1} \rangle_{\mathcal{O}}}^{\varepsilon t}(z^{0} )  \Bigr) \leq c \, \varepsilon,
\end{equation}
for all $\varepsilon\in (0,\varepsilon_{0}]$ and $t\in [0,{L_{0}}/{\varepsilon}]$.
\end{theorem}
If $\varepsilon \rightarrow 0$, we have $\rho\circ\mathrm{Fl}_{\mathbf{X}_{\varepsilon}}^{t} (m^{0})\rightarrow \rho(m^0) $ and $\mathrm{Fl}_{\langle X_{1} \rangle_{\mathcal{O}}}^{\varepsilon t}(z^{0} )=\mathrm{Fl}_{\varepsilon\langle X_{1} \rangle_{\mathcal{O}}}^{ t}(z^{0} )\rightarrow z^0 $. Therefore,  the left hand side of (\ref{CCC}) is identically zero and the conclusion of the Theorem remains true for $\varepsilon =0$ which corresponds to  the unperturbed case.

In fact, we can calculate constant $c$ in (\ref{CCC}) and its value depends only on the Riemannian structure of $M$,
the $\mathbb{S}^{1}$-action induced by $X_0$ and the choice of the open domain $\mathcal{D}_0$ (see Corollary
\ref{cor_cons} and Remark \ref{remest}).

\begin{example}[\textbf{One-frequency systems}]
Let us consider $M = \mathbb{S}^{1} \times \mathbb{R}^k $ with the usual angular coordinate
$\varphi\, (\mathrm{mod} \, 2\pi)$ and $\mathbf{x}=(x_1,x_2,\ldots,x_k)\in \mathbb{R}^k $. Let
${X}_{\varepsilon} = X_{0} +\varepsilon X_{1}$ be a perturbed vector field where
\begin{equation*}
X_0 = \omega(\mathbf{x})\frac{\partial}{\partial \varphi}, \qquad
X_1 = f (\mathbf{x},\varphi)\frac{\partial}{\partial \varphi} + g_{i} (\mathbf{x},\varphi) \,
\frac{\partial}{\partial x_i},
\end{equation*}
with frequency function $\omega >0$ and $f=f (\mathbf{x},\varphi)$, $g_i=f (\mathbf{x},\varphi)$, $2\pi$-periodic
functions in $\varphi$. The vector field ${\partial}/{\partial\varphi}$ is the infinitesimal generator of
the $\mathbb{S}^{1}$-action induced by $X_0$. The orbit space $\mathcal{O}$ can  be identified with $\mathbb{R}^k$,
$\mathcal{O}\cong \mathbb{R}^k$. The average of $X_1$ is given by
\begin{equation*}
\langle X_1 \rangle = \langle f \rangle (\mathbf{x})\frac{\partial}{\partial \varphi} + \langle g_i \rangle
(\mathbf{x})\,\frac{\partial}{\partial x_i},
\end{equation*}
and its reduced vector field is
$\langle X_1 \rangle_\mathcal{O}=\langle g_i \rangle (\mathbf{x})\,{\partial}/{\partial x_i}$. If a trajectory
of $\langle X_1 \rangle_\mathcal{O}$ remains, over time $t = {T}/{\varepsilon}$, on an open domain
$A \subset \mathbb{R}^k$ having compact closure, then {\upshape Theorem \ref{aver_theo}} reduces to the classical
result of averaging for a single-frequency system {\upshape \cite{Arno-63,ArKN-88,Ne-08,verhulst3}}.
\end{example}

\section{$\mathbb{S}^{1}$-invariant normalization of perturbed vector fields}
\label{sec:4}

In the proof of the classical averaging theorem \cite{Arno-63,ArKN-88,verhulst3}, an important role  is played by coordinate changes taking
the original perturbed vector field ${X}_\varepsilon$ into its $\mathbb{S}^{1}$-invariant normal form of first
order, (see Definition \ref{firstorder}) . More precisely, such a change of coordinates allow us to replace the vector field $X_1$ by the
$\mathbb{S}^{1}$-invariant vector field $ \langle X_1 \rangle$ together with a small perturbation (of order
$\varepsilon^2$). The change of coordinates that needs to be performed belongs to the class of \emph{near
identity transformations}. A precise definition and a procedure to construct such a kind of transformations is
given in what follows.

For a nonempty open domain $N \subset M $ and a constant $\delta > 0$ a \emph{near identity transformation} is a
smooth mapping $\Phi: (-\delta,\delta) \times N \rightarrow M$ such that for every $\varepsilon \in (-\delta,\delta)$,
the map defined by $
\Phi_\varepsilon(x):=\Phi(\varepsilon,x)$ is a diffeomorphism onto its image and $\Phi_0 \equiv \mathrm{id}$.

Near identity transformations have the following important property: for a perturbed vector field
${X}_\varepsilon = X_0 + \varepsilon X_1$, the pullback $\Phi_\varepsilon^{\ast} {X}_\varepsilon$ is again a
perturbed vector field whose unperturbed part is $X_0$,
\begin{equation*}
\Phi_{\varepsilon}^{\ast} {X}_{\varepsilon} \Big|_{\varepsilon=0} = X_0.
\end{equation*}
By the Flow Box Theorem \cite{marsden2}, for any open domain $N \subset M$ with compact closure and for an arbitrary smooth
vector field $Z$ on $M$ there exists $\delta >0$ such that the mapping
$\Phi: (-\delta,\delta) \times N \rightarrow M$ given by
\begin{equation}  \label{NI}%
\Phi_{\varepsilon} := \mathrm{Fl}_{Z}^{t} \Big|_{t=\varepsilon}
\end{equation}
is a near identity transformation.

Now we deal with the problem of $\mathbb{S}^{1}$-invariant normal forms of perturbed vector fields.
\begin{definition}\label{firstorder}
Let $X_{\varepsilon} = X_0 + \varepsilon X_1 + O(\varepsilon^2)$ be a perturbed vector field on $M$. It is said that
$X_{\varepsilon}$ is in $\mathbb{S}^{1}$-\emph{invariant normal form of first order} if $X_1$ is
$\mathbb{S}^{1}$-invariant. Moreover, ${X}_\varepsilon$ admits a global $\mathbb{S}^{1}$-normalization of first order
if for any open domain $N$, with compact closure, there exists $\delta > 0$ and a near identity transformation
$\Phi_\varepsilon : (-\delta,\delta) \times N \rightarrow M$ which brings ${X}_\varepsilon$ to a
$\mathbb{S}^{1}$-invariant normal form of first order. In this case, $\Phi_\varepsilon$ is called a
\emph{normalization transformation}.
\end{definition}

If we take the vector field
\begin{equation}    \label{hom_sol}%
Z = \frac{1}{\omega} \mathcal{S}(X_{1}) + \frac{1}{\omega^{3}} \mathcal{S}^{2}(\mathcal{L}_{X_{1}} \omega) X_{0},
\end{equation}
and consider the near identity transformation (\ref{NI}), we get the following $\mathbb{S}^1$-invariant normal
form result.

\begin{proposition} \label{pronorm}
Let $X_{0}$ and $X_{1}$ be vector fields on $M$. If $X_0$ has periodic flow, then for any open domain $N \subset M$, having compact closure, there exists a constant $\delta > 0$ such that the
near identity transformation {\upshape (\ref{NI})} given by the flow  of vector field $Z$ in {\upshape (\ref{hom_sol})},
brings ${X}_\varepsilon = X_0  + \varepsilon X_1$ into the $\mathbb{S}^{1}$-invariant
normal form
\begin{equation*}
\Phi_\varepsilon^{\ast} {X}_\varepsilon = X_0 + \varepsilon \langle X_1 \rangle + \varepsilon^2 R_{\varepsilon},
\end{equation*}
where $R= R_{\varepsilon}$ is a vector field on $M$, smoothly depending on $\varepsilon$.
\end{proposition}
\begin{proof}
By using the non-canonical Lie transform method and Deprit's diagram \cite{Dep-69}, we can compute the expansion of
$(\Phi_{\varepsilon})^{\ast} {X}_{\varepsilon}$ up to any order in $\varepsilon$.
In particular, we have
\begin{equation}   \label{Aeps_1}%
(\Phi_{\varepsilon})^{\ast} {X}_{\varepsilon} = X_{0} + \varepsilon (X_1 - \Li_{ X_0}Z  ) + O(\varepsilon^{2}),
\end{equation}
where $Z$ is the infinitesimal generator of the mapping $\Phi_{\varepsilon}$, as in   (\ref{NI}). It
follows that $\Phi_{\varepsilon}$ is a normalization transformation if
and only if there exist vector fields $W$ and $Z$ satisfying the homological type equation
\begin{equation}  \label{eq:homtype}%
\Li_{X_0} Z = X_{1} - W,
\end{equation}
where $W$ is $\mathbb{S}^{1}$-invariant. Since $X_0$ has periodic flow, Proposition \ref{homol_prop} implies that
vector fields $W=\langle X_1 \rangle $ and
$Z = \frac{1}{\omega}\mathcal{S}(X_{1}) + \frac{1}{\omega^{3}} \mathcal{S}^{2}(\mathcal{L}_{X_{1}} \omega)X_{0}$ are
solutions of equation (\ref{eq:homtype}).
\end{proof}

\section{Gronwall's type estimations on Riemannian  manifolds}
\label{sec:5}

Suppose that  $\varphi: I \subset \mathbb{R} \rightarrow \mathbb{R}$ is a continuous function such that for
$t_{0}\leq t\leq t_{0}+L$, we have
\[
\varphi(t)\leq\delta_{2}(t-t_{0})+\delta_{1}\int_{t_{0}}^{t}\varphi(\tau
)d\tau+\delta_{3},
\]
with constants $\delta_{1} > 0$ and $\delta_{2}\geq 0$, $\delta_{3}\geq0$. The well known Gronwall's lemma asserts
that
\begin{equation}  \label{specGron}%
\varphi(t) \leq\left( \frac{\delta_{2}}{\delta_{1}} + \delta_{3} \right) \mathrm{e}^{\delta_{1} (t-t_{0})}
- \frac{\delta_{2}}{\delta_{1}},
\end{equation}
holds for $t_{0} \leq t \leq t_{0} + L$, \cite{verhulst3}. By using this fundamental inequality, it is possible to
get some estimates for the time evolution of the distance between points of trajectories of two vector fields on
a general Riemannian manifold.

Let $(M, g)$ be a connected Riemannian manifold. Denote by $\mathrm{dist} : M \times M
\rightarrow \mathbb{R}$ the distance function induced by the Riemannian metric $g$. For a
submanifold $N\subset M$, $\mathrm{dist}_N$ will denote the restriction of the distance function to $N$.

Let $\nabla: \mathfrak{X}(M)\times \mathfrak{X}(M)\rightarrow \mathfrak{X}(M) $ be the Levi-Civita connection
associated to $(M, g)$. From the basic properties of the Levi-Civita connection,
$(\nabla_Y X)(m)$ depends only on the value of $Y$ at $m$. Therefore, for each $X\in\mathfrak{X}(M)$ and $m\in M$
we have a linear map $ (\nabla X)_m : (T_{m} M, g_m) \rightarrow (T_{m} M, g_m)$,
$Y(m)\mapsto \nabla_{Y(m)}X$.

In what follows, $\nabla X$ will denote the covariant derivative
of the vector field $X$ and $\| (\nabla X)_m \|$ the operator norm  defined by
\begin{equation*}
\| (\nabla X)_m \| := \sup \{\ \| \nabla_{Y(m)} X \|_m\ \ :\ Y \in \mathfrak{X}(M) \text{ and } \| Y(m) \|_m =1   \}
\end{equation*}
where $\| Y \|_m := g_m (Y,Y). $
Given a diffeomorphism $\varphi: M\rightarrow  M$, the push-forward $\varphi_{\ast}\nabla : \mathfrak{X}( M) \times
\mathfrak{X}( M) \rightarrow \mathfrak{X}( M)$ of the covariant derivative $\nabla$ is defined by
\[
(\varphi_{\ast}\nabla)_{\varphi_{\ast}Y}\varphi_{\ast}X=\varphi_{\ast}(\nabla_{Y}X)
\]
for all $X,Y\in\mathfrak{X}( M)$. Hence, if $\varphi$ is an isometry, then $\varphi$ preserves the connection $\nabla$, that is,
$\varphi_{\ast}\nabla=\nabla$.

We  prove the following technical fact.
\begin{lemma}  \label{lemisom}%
Let $\varphi$ be an isometry on the Riemannian manifold $( M, g)$ and let $X \in\mathfrak{X}( M)$ be
a vector field. Then, $\| \varphi_{\ast}X \|_{m} =\| X \|_{\varphi^{-1}(m)}$ and for the vector bundle morphisms
$\nabla(\varphi_{\ast}X)$ and $\nabla X$ we have
\begin{equation}
\|\nabla(\varphi_{\ast}X)\|_{m}=\|\nabla X\|_{\varphi^{-1}(m)}. \label{iso}
\end{equation}
\end{lemma}
\begin{proof}
Since $\varphi$ is an isometry, we have that $\nabla_{\varphi_{\ast}Y}\varphi_{\ast}X=\varphi_{*}(\nabla_{Y}X)$ and hence
\[
\nabla_{v} (\varphi_{\ast}X) = (\mathrm{d}_{\varphi^{-1}(m)}\varphi) ( \nabla_{\mathrm{d}_{m}\varphi^{-1}(v)}X ),
\]
for every $v\in T_{m} M$. It follows that
\[
\| (\nabla \varphi_{\ast}X)(v) \|_{m} =\| \nabla X(\mathrm{d}_{m} \varphi^{-1}(v)) \|_{\varphi^{-1}(m)}.
\]
This last equality together with $\| \mathrm{d}_{m}\varphi^{-1}(v) \|_{\varphi^{-1} (m)} = \| v \|_{m}$ implies (\ref{iso}).
\end{proof}

Now, we will construct a parameterized surface $(t,s) \mapsto \gamma(t,s)$ on a manifold $M$, which is generated by
trajectories of a parameter dependent family of vector fields, as follows. Let $X_s$ be a one-parameter vector field on $M$, smoothly depending on the parameter $s\in [0,1]$. $\mathrm{Fl}_{X_{s}}^{t}$ denotes the flow of $X_s$ for each $s$. Let $\beta : [0,1] \rightarrow M$ be a parameterized smooth curve on $M$. Then, one can fix $L > 0$ such that for each $s\in [0,1]$, the trajectory $t\mapsto\mathrm{Fl}_{X_{s}}^{t} (\beta(s))$
is defined for all $t \in [0,L]$. The resulting parameterized surface is given by
\begin{eqnarray}
\gamma&:& [0,L ]\times [0,1] \rightarrow M \nonumber \\
\gamma(t,s) &:=& \mathrm{Fl}_{X_{s}}^{t} ( \beta(s) ). \label{FLC}%
\end{eqnarray}
Notice that $\mathrm{Fl}_{X_{s}}^{0} ( \beta(s) ) = \beta(s)$, thus,  for each $t$, the $s$-curve $s \mapsto \gamma_{t}(s) =\gamma(t,s)$ can be viewed as the time evolution of
the ``initial'' curve $\beta(s)$ under the flow of $X_{s}$.

\begin{proposition}  \label{propgron}
The length $L(t)$ of the $s$-curve $s\mapsto\gamma_{t}(s)$ on the parameterized surface {\upshape (\ref{FLC})} satisfies the Gronwall's type estimate
\begin{equation}  \label{Gron1}%
L(t) \leq \left(  \frac{C_{2}}{C_{1}} + L(0) \right)  \mathrm{e}^{C_{1}t} - \frac{C_{2}}{C_{1}},
\end{equation}
for all $t\in [0,L]$. Here
\[
C_{1} = \sup_{\substack{m\in\gamma([0,T]\times[0,1])\\s\in[0,1]}} \left \Vert (\nabla X_{s})_{m} \right \Vert \qquad  \text{and} \qquad
C_{2}=\sup_{\substack{t\in[0,T]\\s\in[0,1]}} \left \Vert \frac{\mathrm{d} }{ \mathrm{d} s} X_{s} \right \Vert _{\gamma(t,s)}.
\]
\end{proposition}
\begin{proof}
\textcolor[rgb]{0.00,0.00,1.00}{It follows directly from the inequality $ \left|\frac{\partial}{\partial t}\|\partial \gamma/\partial s \| \right| \leq \|\nabla_{\frac{\partial \gamma}{\partial t}}(\partial \gamma/\partial s )  \| $ that}
\begin{equation*}
\left\|\frac{\partial \gamma}{\partial s}(t,s) \right\|_{\gamma(t,s)}\leq \left\|\frac{\partial \gamma}{\partial s}(0,s) \right\|_{\gamma(0,s)} +
\int_{0}^{t} \left\|\nabla_{\frac{\partial \gamma}{\partial t}} \bigl( {\partial \gamma}/{\partial s}(\textcolor[rgb]{0.00,0.00,1.00}{t'},s) \bigr) \right\|_{\gamma(\textcolor[rgb]{0.00,0.00,1.00}{t'},s)} \mathrm{d}\textcolor[rgb]{0.00,0.00,1.00}{t'}.
\end{equation*}
\textcolor[rgb]{0.00,0.00,1.00}{Taking into account that the Levi-Civita connection is torsion free and} $\displaystyle L(t) = \int_{0}^{1} \left\| \frac{\partial}{\partial s} \gamma(t,s) \right\|_{\gamma_{t}(s)} \mathrm{d} s$, integration in $s$ gives
\begin{equation}  \label{long}%
L(t) \leq  L(0) + \int_{0}^{t} \int_{0}^{1} \left\| \nabla_{\frac{\partial\gamma}{\partial s}} \frac{\partial\gamma}{\partial t^{\prime}} \right\|_{\gamma(t^{\prime},s)}
\mathrm{d}s \, \mathrm{d}t^{\prime}.
\end{equation}
Furthermore, for every $t\in [0,L]$ the vector field $\frac{\partial }{\partial t} \gamma (t,s) $ satisfies the relation
\[
\nabla_{\frac{\partial \gamma}{\partial s}} ( {\partial \gamma}/{\partial t})  =\left( \frac{\mathrm{d}}{\mathrm{d} s} X_{s}\right)(\gamma(t,s)) +
(\nabla X_s)_{\gamma(t,s)} \left( {\partial\gamma}/{\partial s} \right),
\]
and thus, we get
\[
\left\| \nabla_{\frac{\partial\gamma}{\partial s}} \bigl( {\partial\gamma}/{\partial t} \bigr) \right\|_{\gamma(t,s)} \leq
\left\| \frac{\mathrm{d}}{\mathrm{d} s}  X_{s} \right\|_{\gamma(t,s)} + \left\| (\nabla X_{s})_{\gamma(t,s)} \right\| \left\|
\frac{\partial\gamma}{\partial s} \right \|_{\gamma(t,s)}.
\]
Putting this inequality into (\ref{long}), we get
\[
L(t) \leq L(0) + C_{1} \int_{0}^{t} L(\textcolor[rgb]{0.00,0.00,1.00}{t'}) \, \mathrm{d} \textcolor[rgb]{0.00,0.00,1.00}{t'} + C_{2} t.
\]
Now, applying the usual Gronwall's lemma to last inequality leads to (\ref{Gron1}).
\end{proof}

\section{Riemannian submersions on the $\mathbb{S}^{1}$-principal bundle $( M,\rho,\mathcal{O})$}
\label{sec:6}

Let $(M, g)$ be a connected Riemannian manifold and let $X_0$ be a vector field with periodic
flow. Assume  that the Riemannian metric $g$ is \textit{invariant} with respect to the
$\mathbb{S}^{1}$-action induced by $X_0$, and the flow $\mathrm{Fl}_{\mathrm{\Upsilon}}^{t}$ is an isometry on
$( M, g)$, where $\mathrm{\Upsilon}$ is the infinitesimal generator of the $\mathbb{S}^{1}$-action induced by $X_0$, (\ref{infgen}). If the
$\mathbb{S}^{1}$-action on $M$ induced by $X_0$ is free, then the triple $( M,\rho,\mathcal{O})$ is an
$\mathbb{S}^{1}$-principal bundle. The \emph{vertical subbundle} $\mathbb{V} := \mathrm{ker \, d} \rho$ coincides
with the one dimensional distribution given by $\mathscr{D}:= \{\mathscr{D}_m \subset T_m M  | \mathscr{D}_m = \mathrm{Span} \{ \mathrm{\Upsilon}(m)  \} \}$. In this case, we choose the \emph{horizontal subbundle}
as the orthogonal complement to $\mathbb{V}$, $\mathbb{H} = \mathbb{V}^\bot$.

Since the Riemannian metric is $\mathbb{S}^{1}$-invariant, the horizontal subbundle is also invariant with respect
to the $\mathbb{S}^{1}$-action,
\[
(\mathrm{d}_{m} \mathrm{Fl}_{\mathrm{\Upsilon}}^{t}) (\mathbb{H}_{m}) = \mathbb{H}_{\mathrm{Fl}_{\mathrm{\Upsilon}}^{t}(m)}, \qquad
\forall \; m\in M.
\]
Thus, we have the $\mathbb{S}^{1}$-invariant orthogonal splitting $T M = \mathbb{H} \oplus \mathbb{V}$, and every
vector field $Y$ on $ M$ decomposes into its horizontal and vertical parts, as
\[
Y = Y^{\mathrm{hor}} + Y^{\mathrm{vert}}.
\]
It is clear that the restriction of the differential $\mathrm{d}_{m}\rho : T_{m} M \rightarrow T_{\rho(m)} \mathcal{O}$
to $\mathbb{H}_{m}$ is an isomorphism. Hence, for every vector field $v\in\mathfrak{X}(\mathcal{O})$ there exists a
unique vector field $\mathrm{hor}(v) \in \mathfrak{X}(M)$, called the \emph{horizontal lift} of $v$, which is tangent
to $\mathbb{H}$ and $\mathrm{d}\rho\circ \mathrm{hor} (v) = v \circ \rho$.

Let $g^{\mathcal{O}}$ be the unique Riemannian metric on the orbit space $\mathcal{O}$ such that
the projection $\rho$ is a Riemannian submersion (see \cite{gallot,michor}),
\begin{equation}  \label{PRI}%
g_{m}( u_{1},u_{2}) = g_{\rho(m)}^{\mathcal{O}}( ( \mathrm{d}_{m}\rho ) u_{1}, ( \mathrm{d}_{m} \rho ) u_{2})
\end{equation}
for any $m\in M$ and $u_{1}, u_{2}\in\mathbb{H}_{m}$. Denote by $\mathrm{dist}^{\mathcal{O}} : \mathcal{O} \times
\mathcal{O} \rightarrow \mathbb{R}$,
the distance function associated to the Riemannian metric $g^{\mathcal{O}}$ on $\mathcal{O}$, and
by $\nabla$ and $\nabla^{\mathcal{O}}$ the Levi-Civita connections on the Riemannian manifolds
$( M, g)$ and $(\mathcal{O}, g^{\mathcal{O}})$, respectively.

\begin{lemma}\label{lemma2}
Let $\gamma : [0,1] \rightarrow M$, $s \mapsto\gamma(s)$, be a smooth curve on $ M$ and let
$\alpha:=\rho\circ\gamma : [0,1] \rightarrow \mathcal{O}$, $s\mapsto\rho (\gamma(s))$, be its projection onto the
orbit space. Let $(\frac{\mathrm{d} \gamma}{\mathrm{d} s}  )^{\mathrm{hor}} \in\mathbb{H}_{\gamma(s)}$ and
$(\frac{\mathrm{d} \gamma}{\mathrm{d}s} )^{\mathrm{vert}} \in \mathbb{V}_{\gamma(s)}$ be, respectively, the
horizontal and vertical components in the orthogonal decomposition
\begin{equation}  \label{HVD}%
\frac{\mathrm{d} \gamma}{\mathrm{d} s}  = \left(\frac{\mathrm{d} \gamma}{\mathrm{d} s} \right)^{\mathrm{hor}} +
\left( \frac{ \mathrm{d} \gamma}{\mathrm{d} s} \right)^{\mathrm{vert}}.
\end{equation}
Then,
\begin{itemize}
\item[\upshape (a)] The arc lengths $L(\gamma)$ and $L(\alpha)$ of the curves $\gamma$ and $\alpha$, respectively,
satisfy the inequalities
\begin{equation}  \label{Ner1}%
L(\alpha) \leq L(\gamma),
\end{equation}
\begin{equation}  \label{Ner2}%
L(\gamma) \leq L(\alpha) + \int_{0}^{1} \left\| \left(\frac{\mathrm{d} \gamma}{ \mathrm{d} s} \right)^{\mathrm{vert}}
\right\| \mathrm{d} s \leq
\sqrt{2} L(\gamma).
\end{equation}
The equality $L(\alpha) = L(\gamma)$ holds if and only if the curve $\gamma$ is horizontal, that is,
$(\frac{\mathrm{d} \gamma}{\mathrm{d} s})^{\mathrm{vert}} = 0$.
\item[\upshape (b)] For any $p, q\in M$, we have
\begin{equation}  \label{DNER}%
\mathrm{dist}^{\mathcal{O}} \bigl( \rho(p),\rho(q) \bigr) \leq \mathrm{dist}(p,q).
\end{equation}
\end{itemize}
\end{lemma}
\begin{proof}
Part (a) is evident and follows from the relation $\frac{\mathrm{d} \alpha}{\mathrm{d}s} = (\mathrm{d}_{\gamma(s)} \rho)
(\frac{\mathrm{d} \gamma}{\mathrm{d}s})^{\mathrm{hor}}$, the orthogonal decomposition (\ref{HVD}) and the equality
\begin{equation} \label{COD1}%
\left\|\frac{\mathrm{d}\alpha}{\mathrm{d} s} \right\|^{\mathcal{O}} = \left\|\left(
\frac{\mathrm{d} \gamma}{\mathrm{d} s} \right)^{\mathrm{hor}} \right\|,
\end{equation}
which is a consequence of the property that $\rho$ is a Riemannian submersion. In order to prove part (b), for
arbitrary $p, q\in M$, let us choose a curve $\gamma$ on $ M$ joining points $p$ and $q$ and such that
$\mathrm{dist} (p,q) + \Delta \geq L(\gamma)$, for some $\Delta > 0$. Then, by (\ref{Ner1}) we get
\[
\mathrm{dist}^{\mathcal{O}} (\rho(p),\rho(q))   \leq L(\rho\circ\gamma) \leq L(\gamma) \leq \mathrm{dist} (p,q) + \Delta.
\]
Since $\Delta >  0$ is arbitrary, inequality (\ref{DNER}) is satisfied.
\end{proof}

Now, let $\beta : [a,b] \rightarrow \mathcal{O}$ be a smooth curve on $\mathcal{O}$ passing trough the point
$\beta(a) = x \in \mathcal{O}$. Let $m \in \rho^{-1}(x)$ be a point in the fiber over $x$. A \textit{lifting} of
$\beta$ trough $m$ is a smooth curve $\widetilde{\beta} : [a,b]\rightarrow  M$ such that
\begin{enumerate}
\item[(i)] $m =\widetilde{\beta}(a)$, and
\item[(ii)] $\rho \circ \widetilde{\beta} = \beta$.
\end{enumerate}
In this case, $\widetilde{\beta}$ is called the lift of curve $\beta$ or lifted curve.
A lifted curve $\widetilde{\beta}$ is called \textit{horizontal} if, in addition, it satisfies the following property:
\begin{equation}  \label{horlift}
\frac{\mathrm{d}}{\mathrm{d} t} \widetilde{\beta}(t) \in \mathbb{H}_{\widetilde{\beta}(t)}, \qquad
\forall \; t \in [a,b].
\end{equation}
Since $\mathbb{S}^{1}$ in compact, the fibers are compact. Hence, it follows that for any smooth curve
$\beta : [a,b] \rightarrow \mathcal{O}$, the horizontal lift of $\alpha$ trough $m$ always exists \cite{michor}.
The following statement gives us a key property for the horizontal lift.

\begin{proposition} \label{proplift}
Let $X\in\mathfrak{X}( M)$ be a vector field and $\gamma : [0,T(m^0)]\rightarrow M$ the trajectory
of $X$ through $m^0 \in M$, $\gamma (t) =\mathrm{Fl}_{X}^{t} (m^0)$. Consider the projection
$\alpha = \rho \circ\gamma$ and its horizontal lift $\widetilde{\alpha} : [0,T] \rightarrow M$,
$t \mapsto \widetilde{\alpha} (t)$ through $m^0$,  $\widetilde{\alpha} (0) = m^0$. Then, there exists a smooth
function $\tau : [0,T(m^0)] \rightarrow \mathbb{R}$ such that
$\tau (0) = 0$ and
\begin{equation}  \label{KEY1}%
\widetilde{\alpha} (t) = \varrho^{t} (\gamma(t)),
\end{equation}
where
\begin{equation}  \label{KEY2}%
\varrho^{t} = \mathrm{Fl}_{\mathrm{\Upsilon}}^{\tau(t)}.
\end{equation}
Moreover, the curve $t\mapsto \widetilde{\alpha}(t)\in M$ is the trajectory through $m^0$ of the horizontal
$t$-dependent vector field
\begin{equation}  \label{KEY3}%
\widetilde{X}_{t} = (\varrho^{t})_{\ast}X^{\mathrm{hor}},
\end{equation}
that is,
\[
\frac{\mathrm{d} \widetilde{\alpha}(t)}{ \mathrm{d}t } = \widetilde{X}_{t} (\widetilde{\alpha}(t)),
\]
and the following properties hold:
\begin{equation}  \label{KEY4}%
\| \widetilde{X}_{t} \|_{\widetilde{\alpha} (t)} = \, \|  X^{\mathrm{hor}} \|_{\gamma(t)} ,
\end{equation}
and
\begin{equation}  \label{KEY5}%
\| \nabla_{v} \widetilde{X}_{t} \|_{\widetilde{\alpha}(t)} \; \leq \; \| \nabla X^{\mathrm{hor}} \|_{\gamma(t)} \cdot
\| v \|_{\widetilde{\alpha} (t)},
\end{equation}
for every $v \in T_{\widetilde{\alpha}(t)} M$.
\end{proposition}
\begin{proof}
By definition, for each $t$, the points $\widetilde{\alpha}(t)$ and $\gamma(t)$ belong to the same fiber
$\rho^{-1}(\alpha(t))$ and thus they can be joined by a segment of the periodic trajectory of $\mathrm{\Upsilon}$, for
time $\tau =\tau(t)$. Differentiating both sides of (\ref{KEY1}) with respect to $t$ and using decomposition
(\ref{HVD}), we get
\begin{align}
\frac{\mathrm{d} }{\mathrm{d} t} \widetilde{\alpha}(t) &  = (\mathrm{d}_{\gamma(t)}\varrho^{t})
\frac{\mathrm{d}\gamma(t)}{\mathrm{d}t} + \tau^{\prime}(t) \mathrm{\Upsilon}(\gamma(t))  \label{KEY6}\\
&  = (\mathrm{d}_{\gamma(t)} \varrho^{t}) X^{\mathrm{hor}}(\gamma(t)) + (\mathrm{d}_{\gamma(t)}
\varrho^{t})X^{\mathrm{vert}}(\gamma(t)) + \tau^{\prime}(t) \mathrm{\Upsilon}(\gamma(t)).  \nonumber
\end{align}
Remark that the flow of $\mathrm{\Upsilon}$ is an isometry which preserves the splitting of $T M$ into horizontal and
vertical subbundles. Hence, the diffeomorphisms $\varrho^{t}$ have the same properties. From here and the fact that the
velocity $\frac{\mathrm{d}\widetilde{\alpha}(t)}{\mathrm{d} t}$ is a horizontal vector field, we deduce, from
(\ref{KEY6}), the relations
\begin{equation}  \label{KEY7}%
\frac{\mathrm{d} \widetilde{\alpha}(t)}{\mathrm{d}t} = (\mathrm{d}_{\gamma(t)}\varrho^{t}) X^{\mathrm{hor}}(\gamma(t)),
\end{equation}
and
\begin{equation}  \label{Key7a}%
\tau^{\prime}(t) \mathrm{\Upsilon}(\gamma(t)) = - (\mathrm{d}_{\gamma(t)}\varrho^{t})X^{\mathrm{vert}} (\gamma(t)).
\end{equation}
Notice that formula (\ref{Key7a}) defines the function $\tau = \tau(t)$. Putting
$\gamma(t) = (\varrho^{t})^{-1}(\widetilde{\alpha}(t))$ into (\ref{KEY7}) leads to the relation%
\begin{align*}
\frac{\mathrm{d}\widetilde{\alpha} (t)}{\mathrm{d}t}  &  = ( \mathrm{d}_{(\varrho^{t})^{-1}(\widetilde{\alpha} (t))}
\varrho^{t}) X^{\mathrm{hor}}
\bigl( (\varrho^{t})^{-1}(\widetilde{\alpha}(t) ) \bigr) \\
&  = ( \varrho^{t})_{\ast} X^{\mathrm{hor}}( \widetilde{\alpha}(t) ),
\end{align*}
which says that $\widetilde{\alpha}(t)$ is the trajectory through $m^0$ of the vector field $\widetilde{X}_{t}$
in (\ref{KEY3}). Equality (\ref{KEY4}) follows from the property that the differential of $\varrho^{t}$ is a linear isometry
and the representation $\widetilde{X}_{t}(m) = (\mathrm{d}_{m}\varrho^{t}) X \bigl( (\varrho^{t})^{-1}m \bigr)$.
Finally, applying Lemma \ref{lemisom}, we get
\begin{equation*}
\| \nabla_{v} \widetilde{X}_{t} \|_{\widetilde{\alpha}(t)} = \| \nabla_{(\mathrm{d}_{\widetilde{\alpha}(t)}
\varrho^{t})^{-1}v} X^{\mathrm{hor}} \|_{\gamma(t)} \leq \| \nabla X^{\mathrm{hor}} \|_{\gamma(t)} \cdot \| v \|_{\widetilde {\alpha}(t)}%
\end{equation*}
\end{proof}

\section{Proof of main result}
\label{sec:7}

In this section, we present the proof of the periodic averaging theorem (Theorem \ref{aver_theo}), which is done
in several steps. Here, as in the previous sections, we denote by $g$ an
$\mathbb{S}^{1}$-invariant metric on the manifold $M$ and $\mathrm{dist}: M \times M \rightarrow \mathbb{R}$,
the corresponding distance function. Also, let $g^{\mathcal{O}}$ be the unique Riemannian
metric on $\mathcal{O}$ such that the projection $\rho: M\rightarrow \mathcal{O}$ is a Riemannian submersion and
$\mathrm{dist}^{\mathcal{O}} : \mathcal{O} \times \mathcal{O} \rightarrow \mathbb{R} $ its distance function.

Now, take an open domain $\mathcal{D}$ in $\mathcal{O}$ having compact closure and such that
$\mathcal{D}_{0}\subset\mathcal{D}$. Thus, $N_{0}=\rho^{-1}(\mathcal{D}_{0})$ and $N=\rho^{-1}(\mathcal{D})$ are
open domains in $M$, with compact closure. Since for every $m \in N$ the fiber through $m$ is
contained in $N,$  the sets $N_0$ and $N$ are invariant with respect to the $\mathbb{S}^{1}$-action. We also assume that $N$
and $N_0$ are connected.

\noindent\textbf{Step 1} (\textit{Normalization of the perturbed vector field.}) By Proposition \ref{pronorm} there
exists $\delta > 0$ such that the flow of  vector field
\begin{equation} \label{vfieldZ}%
Z = \frac{1}{\omega}\mathcal{S}(X_1) + \frac{1}{\omega^3}\mathcal{S}^2(\Li_{\langle X_1 \rangle} \omega)X_0
\end{equation}
is a near identity transformation, $\Phi : (-\delta, \delta) \times N \rightarrow M$, which takes the vector field
${X}_\varepsilon$ into the $\mathbb{S}^{1}$-invariant normal form,
\[
\Phi_\varepsilon^{\ast} {X}_\varepsilon = X_0 + \varepsilon \langle X_1 \rangle + \varepsilon^2 R_{\varepsilon}.
\]
Now, for each $s\in[0,1]$, define the $(\varepsilon,s)-$dependent vector field
\begin{equation}  \label{TPVF}%
{\widetilde{X}}_{\varepsilon,s} = X_{0 }+ \varepsilon\langle X_{1} \rangle + s \varepsilon^{2} R_{\varepsilon}.
\end{equation}

\noindent\textbf{Step 2} (\textit{Triangle inequality.}) It is easy to see that $\widetilde{X}_{\varepsilon,1}$ is
the $\mathbb{S}^{1}$-invariant normal form of first order of ${X}_\varepsilon$  and
$\widetilde{X}_{\varepsilon,0} =  X_0 + \varepsilon \langle X_1 \rangle$ is an $\mathbb{S}^{1}$-invariant vector
field $\rho$-related with $\varepsilon \langle X_1 \rangle_{\mathcal{O}}$.  By the triangle inequality, we get
\begin{align}
\mathrm{dist}^{\mathcal{O}} \Bigl( \rho \circ \mathrm{Fl}_{X_\varepsilon}^{t}  (m^0), &
\mathrm{Fl}^{\varepsilon t}_{\langle X_1 \rangle_{\mathcal{O}}}(z^0) \Bigr)  \leq
\mathrm{dist}^{\mathcal{O}} \Bigl( \rho \circ\mathrm{Fl}_{X_\varepsilon}^{t} (m^0),\rho \circ
\mathrm{Fl}_{\widetilde{X}_{\varepsilon,1}}^{t} (m_\varepsilon) \Bigr) \nonumber \\
& + \mathrm{dist}^{\mathcal{O}} \Bigl( \rho \circ\mathrm{Fl}_{\widetilde{X}_{\varepsilon,1}}^{t} (m_\varepsilon),
\rho \circ \mathrm{Fl}_{\widetilde{X}_{\varepsilon,0}}^{t} (m^0) \Bigr),  \label{des_trian}%
\end{align}
where $z^0 = \rho(m^0)$ and  $m_\varepsilon = \Phi_\varepsilon^{-1}(m^0)$ can be though as a parameterized curve depending smoothly on $\varepsilon$.

Since $\Phi_\varepsilon$ is a near identity transformation, there exists a constant $\delta_{0} \in (0, \delta]$
such that $m_\varepsilon  \in N_0$ for all $\varepsilon\in [0, \delta_{0}]$.
\begin{lemma} \label{est_1}
Let $[0,\delta_0] \rightarrow  \mathcal{D}_0$  be a parameterized curve on the orbit space
$\mathcal{O}$ given by $ \varepsilon \mapsto \rho(m_\varepsilon)$. Then, the inequality
\begin{equation*}
\mathrm{dist}^{\mathcal{O}} (\rho(m^0),\rho(m_{\varepsilon})) \leq \mathrm{dist}(m^0, m_{\varepsilon})
\leq \varkappa_0 \varepsilon,
\end{equation*}
holds for $\varepsilon \in [0,\delta_0]$, with
\begin{equation}
\varkappa_0 = \sup_{\substack{m\in N_0}} \| Z(m) \|_{m},  \label{bound1}
\end{equation}
where $Z$ is the vector field given by {\upshape (\ref{vfieldZ})}.
\end{lemma}
\begin{proof}
Let $L_\varepsilon$ be the arc length of the the parameterized curve $m_\varepsilon = \Phi_\varepsilon^{-1}(m^0)$.
Taking into account that $\frac{\mathrm{d} \, m_{\varepsilon}}{\mathrm{d}\, \varepsilon} = - Z(m_{\varepsilon})$,
we get
\begin{eqnarray*}
L_{\varepsilon} &=& \int_{0}^{\varepsilon} \left\| \frac{\mathrm{d} \,  m_{\varepsilon^{\prime}}}{ \mathrm{d} \,
\varepsilon^{\prime} } \right\| \,
\mathrm{d} \, \varepsilon^{\prime}\\
&=& \int_{0}^{\varepsilon} \| Z(m_{\varepsilon}) \| \mathrm{d} \varepsilon^{\prime} \leq
\varepsilon \sup_{\substack{m\in N_0}} \| Z(m) \|_{m}.
\end{eqnarray*}
Here we use the fact that $\rho$ is a Riemannian submersion and the properties of the distance function.
\end{proof}
\begin{lemma}  \label{lemmdom}
Consider the $(\varepsilon,s)$-dependent vector field $\widetilde{X}_{\varepsilon, s}$ given in
{\upshape(\ref{TPVF})}. Then, there exist $L_{0} > 0$ and $\varepsilon_{0} \in (0,\delta_{0}]$ such that every
trajectory of $\widetilde{X}_{\varepsilon, s}$ through $m_{s\varepsilon}$,
\begin{equation}
t \mapsto \gamma_{\varepsilon, s} (t) := \mathrm{Fl}_{\widetilde{X}_{\varepsilon, s}}^{t} (m_{s\varepsilon})
\in N,  \label{TR1}%
\end{equation}
is defined for $t\in[0,{L_{0}}/{\varepsilon}]$ if $\varepsilon
\in (0,\varepsilon_{0}]$ and $s\in[0,1].$
\end{lemma}
\begin{proof}
By direct computations, we have
\begin{equation*}
\left.\frac{\mathrm{d}}{\mathrm{dt} }\right|_{t=0}\left(\mathrm{Fl}_{X_{0}}^{t} \circ
\mathrm{Fl}_{P_{t}}^{\varepsilon t}(m) \right) =\widetilde{X}_{\varepsilon,s}(m), \ \ \forall \ m
\in M,
\end{equation*}
where
\[
P_{t} = ({P}_{\varepsilon, s})_{t} = \langle X_{1} \rangle - t (\Li_{\langle X_{1} \rangle} \omega) \mathrm{\Upsilon}
+ \varepsilon s
(\mathrm{Fl}_{X_{0}}^{t})^{\ast} R_{\varepsilon}
\]
is a time-dependent vector field  parameterically depending on $(\varepsilon, s)$ in a smooth way. Hence, we have the following identity
\begin{equation}
\mathrm{Fl}_{\widetilde{X}_{\varepsilon,s}}^{t} = \mathrm{Fl}_{X_{0}}^{t} \circ
\mathrm{Fl}_{P_{t}}^{\varepsilon t}.  \label{EF}
\end{equation}
Since the flow
of $X_{0}$ is periodic, it is enough to show that for small enough $\varepsilon$, there exists a fixed interval
$[0,L_{0}]$ which belongs to the interval of  definition of the trajectory of ${P}_{t}$ through $m_{s\varepsilon}$.
The vector field $({P}_{0,s})_{t} = \langle X_{1} \rangle - t (\Li_{\langle X_{1} \rangle} \omega) \mathrm{\Upsilon}$
is $\mathbb{S}^{1}$-invariant and $\rho$-related with $\langle X_{1} \rangle_{\mathcal{O}}$. Hence, the trajectory
of this field through $m^0$ is defined for $t\in [0,L_{0}]$, and there exists $\varepsilon_{0} \in (0,\delta_{0}]$
such that for every $\varepsilon \in [0, \varepsilon_{0}]$ and $s\in[0,1]$, the trajectory of
$({P}_{\varepsilon, s} )_{t} = ({P}_{0, s})_{t} + \varepsilon  s (\mathrm{Fl}_{X_{0}}^{t})^{\ast} R_{\varepsilon}$
through $m_{s\varepsilon}$ is also defined for all $t \in [0,L_{0}]$. Here we use , the following well known property
(see \cite{abraham}, page 222): If $[0,L_{0}]$ is contained in the domain of definition of the trajectory through $m^0$,
then there exists a neighborhood $U$ of $m^0$ such that any $m \in U$ has a trajectory existing for time
$t\in [0,L_{0}]$.
\end{proof}

Since ${X}_\varepsilon = (\Phi_\varepsilon)_{\ast} \widetilde{X}_{\varepsilon, 1}$, we get
$\mathrm{Fl}_{\widetilde{X}_{\varepsilon,1}}^{t} (m_\varepsilon) =
\Phi_{\varepsilon}^{-1} \circ \mathrm{Fl}_{X_\varepsilon}^{t} (m^0)$. Taking into account that $\rho$ is a
Riemannian submersion, it follows from Lemma \ref{est_1} that
\begin{eqnarray}
\mathrm{dist}^{\mathcal{O}} (\rho \circ \mathrm{Fl}_{X_{\varepsilon}}^{t} (m^0),\rho \circ
\mathrm{Fl}_{\widetilde{X}_{\varepsilon,1}}^{t} (m_{\varepsilon})) & \leq &
\mathrm{dist} (\mathrm{Fl}_{X_{\varepsilon}}^{t} (m^0),\mathrm{Fl}_{\widetilde{X}_{\varepsilon,1}}^{t}
(m_{\varepsilon})) \nonumber\\
& \leq & \mathrm{dist}(\mathrm{Fl}_{X_{\varepsilon}}^{t} (m^0), \Phi_{\varepsilon}^{-1} \circ
\mathrm{Fl}_{X_{\varepsilon}}^{t} (m^0) )  \nonumber\\
& \leq & \varkappa_{0} \, \varepsilon.       \label{1st_est}
\end{eqnarray}
By Lemma \ref{lemmdom}, estimation (\ref{1st_est}) holds for $t \in [0, {L_{0}}/{\varepsilon}]$,
whenever $\varepsilon \in [0,\varepsilon_{0}]$ and $s\in[0,1]$.\\

\noindent\textbf{Step 3} (\textit{Gronwall's estimations.}) In order to get an estimation of first order in
$\varepsilon $ for the second term in (\ref{des_trian}), we proceed as follows: for each fixed $\varepsilon$ we
define the trajectory $\gamma_{\varepsilon} : [0, {L_0}/{\varepsilon}] \rightarrow N$ of the vector field
$\widetilde{X}_{\varepsilon,1} $ through $m_\varepsilon$, that is, $\displaystyle \gamma_{\varepsilon} =
\mathrm{Fl}_{\widetilde{X}_{\varepsilon,1}}^{t}(m_{\varepsilon})$.

Now, let $\alpha_{\varepsilon} = \rho \circ \gamma_{\varepsilon}$ be the projection of $\gamma_{\varepsilon}$ on
the orbit space and $\widetilde{\alpha}_{\varepsilon}$ the horizontal lift of  $\alpha_{\varepsilon}$ through
$m_{\varepsilon}$ Then, by Proposition \ref{proplift}, for every $t$, there exists a fiberwise diffeomorphism
$\varrho^{t}$ on $N$ defined by (\ref{KEY2}), such that $\varrho^{0} = \mathrm{id}$ and
\[
\widetilde{\alpha}_{\varepsilon}(t) = \varrho^{t} (\gamma_{\varepsilon}(t) ).
\]
Moreover, $\widetilde{\alpha}_{\varepsilon}(t)$ is the trajectory of the time dependent vector field
$(\varrho^{t})_{\ast} \widetilde{X}_{\varepsilon,1}^{\mathrm{hor}}$, where
\[
\widetilde{X}_{\varepsilon,1}^{\mathrm{hor}} = \varepsilon \langle X_{1} \rangle^{\mathrm{hor}} + \varepsilon^{2} R_{\varepsilon}^{\mathrm{hor}}.
\]
Since $\varrho^{t}$ is defined as the reparameterized flow of the infinitesimal generator of the
$\mathbb{S}^{1}$-action, we have that
$(\varrho^{t})_{\ast} \langle X_{1} \rangle^{\mathrm{hor} } = \langle X_{1} \rangle^{\mathrm{hor}}$, and
hence,
\[
(\varrho^{t})_{\ast} \widetilde{X}_{\varepsilon,1}^{\mathrm{hor}} = \varepsilon
\langle X_{1} \rangle^{\mathrm{hor}} + \varepsilon^{2} (\varrho^{t})_{\ast} R_{\varepsilon}^{\mathrm{hor}}.
\]
For every $\varepsilon\in [0,\varepsilon_{0}]$ and $ s\in [0,1]$, consider the following horizontal time dependent
vector field on $N$:
\[
{Y}_{t} (\varepsilon,s) = \varepsilon \langle X_{1} \rangle^{\mathrm{hor}} + s \varepsilon^{2}
(\varrho^{t})_{\ast} R_{\varepsilon}^{\mathrm{hor}}.
\]
We define the following parameterized surface in $N$,
\begin{equation}
\Sigma_{\varepsilon} :  \left[ 0, {L_{0}}/ {\varepsilon} \right] \times [0,1] \ni (t,s) \mapsto
\Sigma_{\varepsilon}(t,s) := \mathrm{Fl}_{Y_{t}}^{t} (m_{\varepsilon s}). \label{SURF}%
\end{equation}
It is clear that
\[
\Sigma_{\varepsilon} (t,0) = \mathrm{Fl}_{\langle X_{1}\rangle^{\mathrm{hor}}}^{\varepsilon t} (m^0).
\]
Since $Y_{t}(\varepsilon,1)$ coincides with $(\varrho^{t})_{\ast} \widetilde{X}_{\varepsilon,1}^{\mathrm{hor}}$
we have that $\displaystyle \widetilde{\alpha}_{\varepsilon}(t) = \Sigma_{\varepsilon}(t,1)$. Thus,
$ \alpha_{\varepsilon}(t)=\rho \circ \Sigma_{\varepsilon}(t,1)$ and $\rho \circ \mathrm{Fl}_{\widetilde{X}_{\varepsilon,0}}^{t} (m^0) = \rho \circ \Sigma_{\varepsilon}(t,0)$ . By construction, we have
\begin{equation}
\mathrm{dist}^{\mathcal{O}} \Bigl( \rho \circ\mathrm{Fl}_{\widetilde{X}_{\varepsilon,1}}^{t} (m_\varepsilon),
\rho \circ \mathrm{Fl}_{\widetilde{X}_{\varepsilon,0}}^{t} (m^0) \Bigr) = \mathrm{dist}^{\mathcal{O}}  \Bigl( \rho \circ \Sigma_{\varepsilon}(t,1), \rho \circ \Sigma_{\varepsilon}(t,0) \Bigr)\label{aux1}
\end{equation}
By part (b) of Lemma \ref{lemma2}, we have the estimation
\begin{equation}
\mathrm{dist}^{\mathcal{O}}  \Bigl( \rho \circ \Sigma_{\varepsilon}(t,1), \rho \circ \Sigma_{\varepsilon}(t,0) \Bigr) \leq
\mathrm{dist} \Bigl( \Sigma_{\varepsilon} (t,1), \Sigma_{\varepsilon}(t,0)  \Bigr).\label{aux2}
\end{equation}
Combining (\ref{aux1}) and (\ref{aux2}) we can get an estimation for the second term in (\ref{des_trian}) by studying the lengths of  the $s$-curves in the surface $\Sigma_{\varepsilon}$ (\ref{SURF}).

Now, for a fixed $t$, consider the horizontal $s$-curve $s\mapsto \Sigma_{\varepsilon,t}(s) := \Sigma_{\varepsilon}(t,s)$ and
its arc length
\[
L_{\varepsilon}(t) := \int_{0}^{1} \left\| \frac{\mathrm{d}}{\mathrm{d} \, s} \, \Sigma_{\varepsilon,t}(s) \right\| \mathrm{d} \, s.
\]
\begin{lemma}  \label{lemgron2}
For all $\varepsilon\in [0,\varepsilon_{0}]$ and $t \in [0,{L_{0}}/{\varepsilon}]$, the following estimate holds:
\begin{equation}
L_{\varepsilon}(t) \leq \left[ \left(  \frac{\varkappa_{2}}{\varkappa_{1} } + \varkappa_{0} \right)
\mathrm{e}^{\varepsilon \varkappa_{1}t} - \frac{\varkappa_{2}}{\varkappa_{1}} \right] \varepsilon,  \label{MMN}
\end{equation}
where $\varkappa_{0}$ is given by {\upshape(\ref{bound1})} and
\begin{align}
\varkappa_{1}  & = \sup_{\substack{m\in \overline{N} \\ \varepsilon\in [0,\varepsilon_{0}]}}
\left\| \nabla \langle X_{1} \rangle^{\mathrm{hor}} \right\|_{m} + \varepsilon \left\|
\nabla R_{\varepsilon}^{\mathrm{hor}} \right\|_{m} ,  \label{CAP3} \\
\varkappa_{2} & =\sup_{\substack{m\in\bar{N}\\ \varepsilon \in [0,\varepsilon_{0}]}}
\| R_{\varepsilon}^{\mathrm{hor}} \|_{m} . \label{CAP4}
\end{align}
\end{lemma}
\begin{proof}
Applying the basic inequality (\ref{long}), we have
\begin{equation}
L_{\varepsilon}(t) \leq  L_{\varepsilon}(0) + \int_{0}^{t} \int_{0}^{1}
\left\| \nabla_{\frac{\partial \Sigma_{\varepsilon}}{\partial s}}
\frac{\partial \Sigma_{\varepsilon}}{\partial t^{\prime}} \right\|_{\Sigma_{\varepsilon}} \mathrm{d} s \,
\mathrm{d}  t^{\prime}.\label{BAS}
\end{equation}
By definition, the $t$-curves in $\Sigma_{\varepsilon}$ are horizontal and
\[
\frac{\partial}{\partial t} \Sigma_{\varepsilon}(t,s) =  \left(\varepsilon \langle X_{1} \rangle^{\mathrm{hor}} + s
\varepsilon^{2} (\varrho^{t})_{\ast} R_{\varepsilon}^{\mathrm{hor}} \right)  \circ \Sigma_{\varepsilon}.
\]
It follows that
\begin{align*}
\left\| \nabla_{\frac{\partial \Sigma_{\varepsilon}}{\partial s}}
\frac{\partial \Sigma_{\varepsilon}}{\partial t^{\prime}} \right\|_{\Sigma_{\varepsilon}}  & \leq
\varepsilon \left\| \nabla \langle X_{1} \rangle^{\mathrm{hor}} \right\|_{\Sigma_{\varepsilon}} \cdot
\left\| \frac{\partial \Sigma_{\varepsilon}}{\partial s} \right\|_{\Sigma_{\varepsilon}} \\
& + s \varepsilon^{2}   \left\| \nabla_{\frac{\partial \Sigma_{\varepsilon}}{\partial s}}
\left(  (\varrho^{t})_{\ast} R_{\varepsilon}^{\mathrm{hor}} \right)
\right\|_{\Sigma_{\varepsilon}} + \varepsilon^{2}  \|  (\varrho^{t})_{\ast}
R_{\varepsilon}^{\mathrm{hor}} \|_{\Sigma_{\varepsilon}}.
\end{align*}
By Lemma \ref{lemisom}, we deduce
\begin{align*}
\left \| \nabla_{\frac{\partial \Sigma_{\varepsilon}}{\partial s} } \left(  (\varrho^{t})_{\ast}
R_{\varepsilon}^{\mathrm{hor}} \right) \right\|_{\Sigma_{\varepsilon}}
& \leq \, \left\| \nabla \left(  (\varrho^{t})_{\ast} R_{\varepsilon}^{\mathrm{hor}} \right)
\right\|_{\Sigma_{\varepsilon}} \cdot \left\| \frac{\partial}{\partial s}   \Sigma_{\varepsilon}
\right\|_{\Sigma_{\varepsilon}} \\
&  = \, \left\| \nabla R_{\varepsilon}^{\mathrm{hor}}  \right\|_{\varrho^{-t} \circ \Sigma_{\varepsilon}}
\cdot \left\|  \frac{\partial}{\partial s} \Sigma_{\varepsilon} \right\|_{\Sigma_{\varepsilon}}
\end{align*}
and
\[
\| (\varrho^{t})_{\ast} R_{\varepsilon}^{\mathrm{hor}} \|_{\Sigma_{\varepsilon}} =
\| R_{\varepsilon}^{\mathrm{hor}} \|_{\varrho^{-t} \circ \Sigma_{\varepsilon}}.
\]
Putting these relations into (\ref{BAS}) we arrive at the inequality%
\[
L_{\varepsilon}(t) \leq \varepsilon \varkappa_{0} + \varepsilon \varkappa_{1}
\int_{0}^{t} L_{\varepsilon}(t^{\prime}) \mathrm{d} t^{\prime} + \varepsilon^{2} \varkappa_{2} t,
\]
and, by applying the specific Gronwall's lemma, we get (\ref{MMN}).
\end{proof}

Finally,  we need to prove that estimation (\ref{CCC}) holds for all
$\varepsilon \in (0, \varepsilon_0]$ and $t \in [0, L_0/\varepsilon].$ This fact concludes the proof of Theorem \ref{aver_theo}. From
estimation (\ref{1st_est}), we have
\begin{equation}
\mathrm{dist}^{\mathcal{O}} (\rho \circ \mathrm{Fl}_{X_{\varepsilon}}^{t} (m^0),\rho \circ
\mathrm{Fl}_{\widetilde{X}_{\varepsilon,1}}^{t} (m_{\varepsilon}))  \leq  \varkappa_{0} \, \varepsilon.\label{aux3}
\end{equation}
By equations (\ref{aux1}), (\ref{aux2}), the inequality $ \mathrm{dist} \Bigl( \Sigma_{\varepsilon} (t,1) ,\Sigma_{\varepsilon}(t,0) \Bigr) \leq L_{\varepsilon} (t)$ and Lemma \ref{lemgron2},
we have
\begin{equation}
\mathrm{dist}^{\mathcal{O}}  \Bigl(\rho \circ \mathrm{Fl}_{\widetilde{X}_{\varepsilon,0}}^{t} (m^0) \Bigr) \leq \left[ \left(  \frac{\varkappa_{2}}{\varkappa_{1} } + \varkappa_{0} \right)
\mathrm{e}^{\varepsilon \varkappa_{1}t} - \frac{\varkappa_{2}}{\varkappa_{1}} \right] \varepsilon.\label{aux4}
\end{equation}
Therefore, the desired result follows from triangle inequality (\ref{des_trian}) and
inequalities (\ref{aux3}) and (\ref{aux4}).

\begin{corollary} \label{cor_cons}
The $\varepsilon$-independent constant in {\upshape (\ref{CCC})} can be chosen as follows:
\begin{equation}
c = \varkappa_{0} + \left(  \frac{\varkappa_{2}}{\varkappa_{1}} + \varkappa_{0} \right)
\mathrm{e}^{\varkappa_{1} T_{0}} - \frac{\varkappa_{2}}{\varkappa_{1}}, \label{valcons}
\end{equation}
where the constants $\varkappa_{0}$, $\varkappa_{1}$, and $\varkappa_{2}$ are given by
{\upshape (\ref{bound1}), (\ref{CAP3})} and {\upshape(\ref{CAP4})}, respectively.
\end{corollary}

\begin{remark} \label{remest}
Taking into account that $\displaystyle \varkappa_0 = \sup_{\substack{m\in N_0}} \| Z(m) \|_{m}$ (Lemma \ref{est_1})  where $Z$ is the vector field given by
\[
Z  = \frac{1}{\omega} X_{1}  + \frac{1}{\omega^{3}}
\mathcal{S}^{2}(\mathcal{L}_{\langle X_{1} \rangle} \omega)   X_{0} ,
\]
$\varkappa_{0}$ is the unique constant of equation (\ref{valcons}) that can be expressed only in terms of vector fields $X_{0}$ and $X_{1}$.
\end{remark}

\section{Application of the averaging theorem to adiabatic invariants}
\label{sec:8}

Here, we present an application of Theorem \ref{aver_theo} in the context of adiabatic invariants which
appear in many important problems of mathematical-physics \cite{Ne-08}. \\

\noindent\textbf{Adiabatic invariants}. An adiabatic invariant of a perturbed vector field is a function which
changes very little along the trajectories of the vector field over a long period of time. More precisely,
let $M$ be a smooth manifold and let $X_0$ be a complete vector field on $M$. A function $I \in C^{\infty}(M)$ is called
an adiabatic invariant of the perturbed vector field $X_{\varepsilon} = X_0 + \varepsilon X_{1}$, if there exists a
constant $c$ such that for every $x\in M$ and $\varepsilon > 0$, the following inequality holds,
\begin{equation*}
| I \circ\mathrm{Fl}_{X_{\varepsilon}}^{t} (x) - I (x) | \leq c \varepsilon, \qquad
\text{for} \; 0 \leq t \leq \frac{1}{\varepsilon}.
\end{equation*}
Now we prove a result that states the conditions for the existence of an adiabatic invariant of a perturbed vector
field $X_\varepsilon = X_0 + \varepsilon X_1$ which is $\varepsilon$-close to a vector field with periodic flow.
Since this proposition relies on Theorem \ref{aver_theo}, the $\mathbb{S}^1$-action induced by $X_0$ must be
free and $M$ a connected manifold.
\begin{proposition}\label{adinv}
Assume also that the reduced averaged vector field $\langle X_{1}\rangle_{\mathcal{O}}$ on the orbit space
$\mathcal{O}= M \diagup \mathbb{S}^{1}$ satisfies the hypothesis of Theorem \ref{aver_theo}  and admits a
smooth first integral $J_{\mathcal{O}}:\mathcal{O}\rightarrow\mathbb{R}$,
\begin{equation}
\mathcal{L}_{\langle X_{1}\rangle_{\mathcal{O}}}J_{\mathcal{O}}=0.\label{condadia}
\end{equation}
Then, the function $J:=J_{\mathcal{O}}\circ\rho$ is an adiabatic invariant for $X_{\varepsilon}$:
\[
| J\circ\mathrm{Fl}_{X_{\varepsilon}}^{t} ( m^0 ) - J (m^0)| = O (\varepsilon),
\]
for $m^0 \in\mathcal{D}_{0}$ , $\varepsilon$ small enough  and $t \in [0,{T_{0}}/{\varepsilon}]$.
\end{proposition}
\begin{proof}
Since the closure of the open domain $\mathcal{D}$ is compact, the function $J_{\mathcal{O}}$  has the Lipschitz property on $\overline{\mathcal{D}}$ (see,
for example, \cite{abraham,verhulst3}),
\[
\left| J_{\mathcal{O}}(z) - J_{\mathcal{O}}(y) \right| \leq \lambda_{J}\left\| z - y \right\|_{\mathcal{O}}.
\]
Then, by condition (\ref{condadia}) and Theorem \ref{aver_theo} we have
\begin{align*}
&  | J \circ \mathrm{Fl}_{X_{\varepsilon}}^{t} (m^0) - J (m^0) | \\
&  = | J_{\mathcal{O}}(\rho \circ \mathrm{Fl}_{\mathbf{X}_{\varepsilon}}^{t} (m^0)) -
J_{\mathcal{O}} (\mathrm{Fl}_{\langle X_{1}\rangle_{\mathcal{O}}}^{\varepsilon t} (\rho(m^0)) | \\
&  \leq \lambda_{J} \left\|  \rho\circ \mathrm{Fl}_{X_{\varepsilon}}^{t} (m^0) -
\mathrm{Fl}_{\langle X_{1}\rangle_{\mathcal{O}}}^{\varepsilon t} (\rho(m^0)) \right\| _{\mathcal{O}}\\
&  \leq\lambda_{J} c\varepsilon,
\end{align*}
where the constant $c$ is given by (\ref{valcons}).
\end{proof}

This result is well known for perturbed Hamiltonian vector fields, where the unperturbed part is a one degree of freedom Hamiltonian system, see \cite{ArKN-88,Ne-08}, and the proof in this case relies on the classical averaging theorem and the existence of action-angle variables, (coordinate approach). A study of the existence of adiabatic invariant for perturbed vector fields with a free coordinate approach can be found in \cite{AvenVallVor-13}.

\begin{acknowledgements}
The authors are grateful to professor Yu. M. Vorobev for fruitful discussions on the preparation of this paper. They also would
like to thank the referees for their useful suggestions that helped us to improve this work. This research was partially supported by the
National Council of Science and Technology (CONACyT) under the Grant 219631, CB-2013-01.
\end{acknowledgements}

\end{document}